\documentclass[11pt,a4paper]{amsart}
\usepackage{amsfonts,amsmath,amssymb,amsthm}
\usepackage[width=16cm,height=23cm]{geometry}
\usepackage{comment}
\usepackage{color}
\usepackage{bbm}
\usepackage{enumitem}
\usepackage{amsmath}
\usepackage[utf8]{inputenc}
\usepackage{tikz} 
\usetikzlibrary{cd}

\newtheorem{definicion}{Definition}[section]
\newtheorem{lema}[definicion]{Lemma}
\newtheorem{proposicion}[definicion]{Proposition}
\newtheorem{teorema}[definicion]{Theorem}
\newtheorem{corolario}[definicion]{Corollary}
\newtheorem{nota}[definicion]{Remark}

\newcommand{\nat}{\mathbb{N}}

\newcommand{\lphi}{{L^{\varphi}(\Omega)}}
\newcommand{\lphiconj}{{L^{\varphi^*}(\Omega)}}
\newcommand{\lpsi}{{L^{\psi}(\Omega)}}

\begin{document}

\setlength{\baselineskip}{16pt}

\title[Weak compactness in nice Musielak-Orlicz spaces]{Weak compactness in nice Musielak-Orlicz spaces}
\author[M. Sanchiz]{Mauro Sanchiz$^{1}$}
\address{Department of Mathematical Analysis and Applied Mathematics, Faculty of Mathematics, Complutense University of Madrid, 28040 Madrid, Spain.}
\email[Mauro Sanchiz]{msanchiz@ucm.es}
\thanks{$^1$Supported by NAWA under the Ulam postoctoral program 2024/1/00064 and partially supported by the research project 2025/00145/001 ``Operadores, retículos y estructura de espacios de Banach''}
\date{}

\begin{abstract}

We prove two weak compactness criteria in Musielak-Orlicz spaces for $N$-functions satisfying the $\Delta_2$-condition. They extend criteria from And\^o for Orlicz spaces to this setting of non-symmetrical Banach function spaces. As consequences, we prove criteria for a sequence in a Musielak-Orlicz space to be weakly convergent, and show that Musielak-Orlicz spaces with the subsequence splitting property are weakly Banach-Saks. The study includes the case of Musielak-Orlicz sequence spaces.

\end{abstract}
\maketitle

\vspace{-10mm}

\section{Introduction}
Compactness has been a central subject of study in functional analysis since the beggining of the field. From the Ascoli-Arzela's theorem for continuous functions to Kolmogorov's for $L_p$ spaces. As compactness is such a strict and difficult property to attain, weaker forms of compactness (and convergence) have been equally studied and its results celebrated in all kind of Banach spaces, such as Dunford-Pettis's Theorem for $L_1$ or Kakutani's Theorem solving the problem for reflexive spaces. Many other works have approached that subject. Recall, for example, the pioneer work of W. Luxemburg \cite{Luxemburg} and T. And\^o \cite{Ando}, or the later examples of Nowak (\cite{Nowak}), Alexopoulos (\cite{Alexopoulos}), Dodds et al. \cite{Dodds} and more recently Le\'{s}nik et al. (\cite{Maligranda}).

In 1962, Andô proved in \cite{Ando} two particularly interesting criteria for the weak compactness of a subset $S$ of an Orlicz space $L^{\varphi}$ for a nice enough Young function $\varphi$ (being a $\Delta_2$ $N$-function) to be weakly compact. First, that $S$ is weakly compact if and only if
$$
\lim_{\lambda\rightarrow 0}\sup_{f\in S} \frac{1}{\lambda}\int_{\Omega} \varphi(\vert f(x)\vert) dx=0,
$$
and second, that $S$ is weakly compact if it is bounded in another $\lpsi$ space with a more restrictive Young function $\psi$ (or the same $\varphi$ if $\lphi$ is reflexive), namely that $\psi$ increases more rapidly than $\varphi$.

These useful criteria stated by And\^o for finite measure spaces $\Omega$ have been extended and improved since then. In \cite{Nowak}, Nowak proved the same results for $\lphi$ spaces for $\sigma$-finite measure spaces $(\Omega,\mu)$. And more recently, similar results have been proved in the context of variable exponent Lebesgue spaces and variable exponent Lebesgue-Boschner sequence spaces in \cite{PCM}, \cite{PCM2}, \cite{PCM3} and \cite{ShiShiWang}.

The Orlicz spaces and the variable exponent Lebesgue spaces are generalizations of the classical Lebesgue spaces $L_p$ which are, in some sense, complementary. The Orlicz spaces $\lphi$ generalize the $L_p$ spaces by considering any Young function $\varphi$ to compute the integral $\int_\Omega \varphi(\vert f(x)\vert)$, not only the exponent $\varphi(t)=t^p$. Meanwhile, the variable exponent Lebesgue spaces $L^{p(\cdot)}(\Omega)$ only take exponent functions, but that can change at each point $x\in\Omega$, computing the integral $\int_\Omega \vert f(x)\vert^{p(x)} dx$, leading to non-symmetrical Banach function spaces. Both classes of spaces are themselves particular cases of the further genralization of Musielak-Orlicz spaces, where the Young function can also change at each point of the space, $\varphi(x,t)$. It was surprising that the Andô criteria for Orlicz spaces also work for variable exponent Lebesgue spaces. But, given that, it is not surprising that those criteria could also be extended to general Musielak-Orlicz spaces.

In this article, we generalize this study of weak compactness and extend Andô criteria to Musielak-Orlicz spaces $\lphi$, expanding the results from \cite{Nowak} and \cite{PCM3}. The two main results are that, under some restrictions on the function $\varphi$, a subset $S$ of $\lphi$ is relatively weakly compact (meaning that its closure $\overline{S}$ is weakly compact) if and only if either (Theorem \ref{crit2})
$$
\lim_{\lambda\rightarrow 0}\sup_{f\in S} \frac{1}{\lambda}\int_{\Omega} \varphi(x,\vert f(x)\vert) dx=0,
$$
or that (Theorem \ref{Th.M-O}) $S$ is norm bounded in a Musielak-Orlicz space $\lpsi$ such that $\psi$ increases uniformly more rapidly than $\varphi$. This is proved in Section \ref{sec.debilcompacidad} after the preliminaries, along with the consequences of criteria for Musielak-Orlicz sequence spaces and that $\lphi$ spaces are weakly Banach-Saks whenever $\lphi$ has the subsequence splitting property besides of the hypothesis of the weak compactness criteria. In Section \ref{sec.wconvergencia}, the weak compactness criteria is applied to sequences to get a variety of conditions on which a sequence $(f_n)$ in $\lphi$ is weakly convergent to some $f\in\lphi$, and the same for Musielak-Orlicz sequence spaces.

\vspace{-5mm}

\section{Preliminaries}\label{sec.preliminaresç}

Throughout the paper $(\Omega,\mu)$ is a $\sigma$-finite separable measurable space and
$L_{0} (\Omega)$ is the space of all real measurable function
classes on $(\Omega,\mu)$. We will use the notation in the monographic \cite{DHHR}, since we take from it almost every property we need of Musielak-Orlicz spaces (all of them well known).
A {\it $\Phi$-function} (or Young function) is a convex and left-continuous function $\varphi:[0,\infty)\rightarrow [0,\infty]$ with $\varphi(0)=0$, $\lim_{t\rightarrow 0^+}\varphi(t)=0$ and $\lim_{t\rightarrow \infty}\varphi(t)=\infty$. A {\it generalized $\Phi$-function} is a function $\varphi:\Omega\times[0,\infty)\rightarrow[0,\infty]$ such that $\varphi(y,\cdot)$ is a $\Phi$-function for every $y\in\Omega$ and $\varphi(\cdot,t)$ is measurable for every $t\geq 0$. Given a generalized $\Phi$-function $\varphi$, the {\it Musielak-Orlicz space (or generalized Orlicz space)} $\lphi$ consists of all measurable scalar function classes $f \in L_0(\Omega)$ such that the modular $\rho_{\varphi}(f/r)$ is finite for some $r>0$, where
$$
\rho_{\varphi}(f):= \int _{\Omega} \varphi(x,\vert f(x)\vert) d\mu(x) < \infty.
$$
To simplify the notation, we will define $\varphi(x,t):=\varphi(x,\vert t\vert)$ for $t<0$ to just write $\varphi(x,f(x))$ instead of $\varphi(x,\vert f(x)\vert)$ when considering possible $f(x)<0$ in the modular above. The associated Luxemburg norm is defined as
$$
\lVert f\rVert_{\varphi} := \inf \left\{ r>0 : \rho_{\varphi}\left(f/r\right)\leq 1 \right\}.
$$
By $B_{\lphi}$ we denote the unit ball of $\lphi$. The Musielak-Orlicz spaces generalize Orlicz spaces by considering different Young functions at each $x\in\Omega$. 
If $\Omega=\mathbb{N}$ with the counting measure, then $\varphi(n,\cdot)=(\varphi_n)$ and $\lphi=\ell_{\varphi_n}$ is called the {\it Musielak-Orlicz sequence space}. We write $x=(x_n)\in \ell_{\varphi}$ and
$$
\lVert(x_n)\rVert_{\varphi_n}=\inf\left\{ r>0: \sum_{n=1}^{\infty}\varphi_n \left(\left\vert\frac{x_n}{r}\right\vert \right)\leq 1 \right\}.
$$

The {\it conjugate function} $\varphi^*$ of $\varphi$ is the generalized $\Phi$-function defined for all $u\geq 0$ and $y\in\Omega$ by the formula $\varphi^*(y,u)=\sup_{t\geq 0}(tu-\varphi(y,t))$. Young's inequality (\cite{DHHR} Eq. 2.6.2) states that, for every $t,u\geq 0$ and $y\in\Omega$,
\begin{equation}\label{Young-ineq}
    t\cdot u\leq \varphi(y,t)+\varphi^*(y,u).
\end{equation}

$\big(\lphi,\lVert\cdot\rVert_{\varphi}\big)$ is a Banach function space with the usual pointwise order if and only if the generalized $\Phi$-function $\varphi$ is {\it proper}, meaning that the simple functions belong to both $\lphi$ and its {\it associate space $(\lphi)'$}, which is the space of functions $g\in L_0(\Omega)$ such that $\sup_{f\in \lphi, \lVert f\rVert_\varphi\leq 1} \int \vert f(x)g(x)\vert d\mu<\infty$ (see \cite{DHHR} page 61). Also in that case, its associate space $(\lphi)'$ coincides with the Musielak-Orlicz space of the conjugate function $\lphiconj$ (\cite{DHHR} Thm 2.7.4). $\varphi$ is said to be {\it locally integrable} if $\rho_{\varphi}(t\chi_{E})<\infty$ for every $t\geq 0$ and $E\subset\Omega$ with $\mu(E)<\infty$.


A generalized $\Phi$-function $\varphi$ satisfies the {\it $\Delta_2$-condition} if there exists $K\geq 2$ such that $\varphi(y,2t)\leq K\varphi(y,t)$ for all $y\in\Omega$ and $t\geq 0$. If $\varphi$ satisfies the $\Delta_2$-condition, then proper implies locally integrable for $\varphi$ (since $\rho_{\varphi}(2^t\chi_E)\leq K^t \rho_{\varphi}(\chi_E)<\infty$) and $\lphi=E^\varphi$ (\cite{DHHR} page 49), where $E^{\varphi}:=\{f\in \lphi: \rho_{\varphi}(f/r)<\infty \text{ for all } r>0\}$. 
In that case, the space has a lot of nice properties: it is separable (\cite{DHHR} Thm 2.5.10), order continuous (see f.i. \cite{Hudzik} page 545), modular convergence (respectively boundedness) and norm convergence (resp. boundedness) are equivalent (\cite{DHHR} page 42) and its dual space coincides with the Musielak-Orlicz space generated by its conjugate function, $(\lphi)^*=L^{\varphi^*}(\Omega)$ (\cite{DHHR} Thm 2.7.14).

A Young function is an {\it $N$-function} if it is continuous, positive, $\lim_{t\rightarrow 0}\frac{\varphi(t)}{t}=0$ and $\lim_{t\rightarrow \infty}\frac{\varphi(t)}{t}=\infty$. A generalized $\Phi$-function is a {\it generalized $N$-function} if $\varphi(y,\cdot)$ is an $N$-function for every $y\in\Omega$. This class of $\Phi$-functions excludes spaces like $L_1(\Omega)$ (by making $\varphi$ more convex than $\psi(t)=t$) and $L_\infty(\Omega)$ (by making $\varphi$ continuous. Also excluded if $\varphi$ satisfies the $\Delta_2$-condition).

Finally, a subset $S$ of $\lphi$ is {\it equi-integrable} in $\lphi$ if, for every decreasing sequence of measurable sets $(A_n)$ in $\Omega$ with zero measure intersection, $\lim_{n\rightarrow\infty} \sup_{f\in S} \, \lVert f\chi_{A_n}\rVert_{\varphi}=0$.
For order continuous Banach function spaces, equi-integrability is equivalent to $S$ satisfying both conditions $\lim_{\mu(A)\rightarrow 0} \sup_{f\in S} \, \lVert f\chi_{A}\rVert=0$ (uniformly integrable) and that, for every $\varepsilon>0$, there exists $A\subset\Omega$ with $\mu(A)<\infty$ such that $\sup_{f\in S} \, \lVert f \chi_{\Omega\setminus A}\rVert<\varepsilon$ (uniform decay at infinity) (cf. \cite{PCM2} Proposition 3.1). 


\section{Weak compactness criteria}\label{sec.debilcompacidad}

W. Luxemburg collected and proved in his thesis the following result for general Banach function spaces $X(\Omega)=X$ on $\sigma$-finite measure spaces $(\Omega,\mu)$ (where $X'$ denotes its associate space or K\"othe dual and $\sigma(X,X')$ the weak topology defined by $X'$).

\begin{teorema}[\cite{Luxemburg} Theorem 5, page 32]
A subset $S$ in $X$ is relatively $\sigma(X, X')$-compact if and only if the mapping on $X'$ defined by
$$
N(g):=\sup_{f\in S} \int_{\Omega} \vert fg\vert d\mu
$$
is an absolutely continuous normal semi-norm.
\end{teorema}


In our setting of Musielak-Orlicz spaces $X=L^\varphi(\Omega)$ with $\varphi$ being proper, the associate space $X'$ is the Musielak-Orlicz space of the conjugate function $X'=L^{\varphi^*(\Omega)}$. Then, $N(g)$ is finite for every $g\in L^{\varphi^*(\cdot)}(\Omega)$ and hence $N(\cdot)$ defined is a semi-norm on $L^{\varphi^*(\cdot)}(\Omega)$ if and only if $S$ is norm bounded (\cite{Luxemburg} page 19). It is clear from the definition that $N(\cdot)$ is a normal semi-norm (see \cite{Luxemburg} page 20). And the property of $N(\cdot)$ be absolutely continuous is equivalent to say that, for every $g\in L^{\varphi^*(\cdot)}(\Omega)$, the set $S_g:=\{ fg: f\in S\}$ is an equi-integrable subset in $L_1(\Omega)$ (see \cite{Luxemburg}, page 22, also \cite{PCM2}). Finally, if $\varphi$ satisfies the $\Delta_2$-condition, then $L^{\varphi^*(\cdot)}(\Omega)=(L^{\varphi}(\Omega))^*$ and the $\sigma(X,X')$ is the weak topology. Thus, we have the first weak compactness criterion:

\begin{teorema}\label{clave}
Let $\lphi$ be a Musielak-Orlicz space for a $\sigma$-finite measure space $(\Omega,\mu)$ and a proper generalized $\Phi$-function $\varphi$ satisfying the $\Delta_2$-condition. A subset $S\in \lphi$ is relatively weakly compact if and only if it is bounded and, for every $g\in L^{\varphi^*}(\Omega)$, both
    \begin{equation}\label{basico}
    \lim_{\mu(E)\rightarrow 0}\sup_{f\in S}\int_E \vert fg\vert d\mu =0
    \end{equation}
    and, for every $\varepsilon>0$, there exists some $A\subset\Omega$ with $\mu(A)<\infty$ such that
    \begin{equation}\label{fuerafinito}
    \sup_{f\in S} \int_{\Omega\setminus A} \vert fg\vert d\mu <\varepsilon.
    \end{equation}
\end{teorema}

\begin{nota}
If $\Omega$ is a non-atomic measure space, the boundedness of $S$ is a redundant hypothesis since it is deduced from (\ref{basico}) and (\ref{fuerafinito}). If not, it is essential (see \cite{PCM3} Rmk 3.3 and below\footnote{Although only proved for variable Lebesgue spaces, the arguments are the same in this case.}).
\end{nota}

To extend And\^o's criteria we need the following notion:

\begin{definicion}
    We say that the generalized $\Phi$-function $\varphi$ is constrained if
    $$
    \lim_{t\rightarrow 0} \sup_{x\in\Omega} \varphi(x,t)=0 \quad \text{ and } \quad \lim_{t\rightarrow \infty} \inf_{x\in\Omega} \varphi(x,t)=\infty.
    $$
\end{definicion}

\begin{nota}
    This notion is a kind of ``uniformity'' for $x$ in the $\Phi$-functions $\varphi(\cdot,t)$ conforming the generalized $\Phi$-function $\varphi(x,t)$. Note that the $\Delta_2$-condition already is an actual ``uniform'' condition, since the constant $K$ does not depend on $x\in\Omega$, while the $N$-funcion condition is not.
\end{nota}

We pass now, after a lemma, to enunciate the main result, a generalization of a criterion due to Andô for Orlicz spaces \cite{Ando}. This criterion only computes the modular of elements $S$ to show if it is weakly compact, instead of requiring to check every element $g\in \lphiconj$. To prove it, we will need the following lemma, which, although almost trivial, requires that $\varphi$ is constrained, and it is the only part of the article where the constrain of $\varphi$ is used. The proof of the theorem is adapted (to Musielak-Orlicz spaces) from those in \cite{Nowak} and \cite{PCM3} (for Orlicz and variable Lebesgue spaces), themselves adaptations of the original in \cite{Ando}. 

\begin{lema}\label{mallema1}
    Let $\lphi$ be a Musielak-Orlicz space for a constrined generalized $\Phi$-function $\varphi$ and a sequence $(f_n)$ in $B_\lphi$. Then, the sets $B_n:=\{x\in\Omega:\vert f_n(x)\vert>n\}$ satisfy $\lim_{n\rightarrow\infty}\mu(E_n)=0$.
\end{lema}

\begin{proof}
    We have
    $$\inf_{x\in B_n}\varphi(x,n)\mu(B_n)
    \leq \int_{B_n} \varphi(x, f_n(x))d\mu\leq 1,
    $$
    so, since $\varphi$ is constrained,
    $$\lim_{n\rightarrow\infty}\mu(B_n)
    \leq \lim_{n\rightarrow\infty}\frac{1}{\inf_{x\in B_n}\varphi (x,n)}
    \leq \lim_{n\rightarrow\infty}\frac{1}{\inf_{x\in \Omega}\varphi (x,n)}= 0.
    $$
\end{proof}

\begin{nota}
    If $\varphi$ is not constrained and $\mu(\Omega)=\infty$ it can be $\mu(B_n)\geq 1$ even if $\varphi$ is an $N$-function with the $\Delta_2$-condition. Take for example $\Omega=(0,\infty)$, $\varphi(x,t)=\frac{1}{x^2}t^{2}$ and $f_n(x)=x\chi_{[n,n+1)}$.
\end{nota}

\begin{teorema}\label{crit2}
Let $\lphi$ be a Musielak-Orlicz space for a $\sigma$-finite measure space $(\Omega,\mu)$ and a proper constrained generalized $N$-function $\varphi$ satisfying the $\Delta_2$-condition. A subset $S\subset \lphi$ is relatively weakly compact if and only if
\begin{equation}\label{condicion}
\lim_{\lambda\rightarrow 0} \sup_{f\in S} \frac{\rho_{\varphi}(\lambda f)}{\lambda}
=\lim_{\lambda\rightarrow 0} \sup_{f\in S} \frac{1}{\lambda}\int_{\Omega} \varphi(x,\lambda \vert f(x)\vert) d\mu=0.
\end{equation}
\end{teorema}

\begin{proof}

$(\Rightarrow)$:
    First, as $\Omega$ is $\sigma$-finite, take $\Omega=\bigcup_{m\in\mathbb{N}} Z_m$ with $(Z_m)\nearrow \Omega$ and $\mu(Z_m)<\infty$. Since $S$ is weakly compact, it is norm bounded and we can suppose w.l.o.g. $S \subset B_{L^{\varphi}(\Omega)}$. Thus, for every $ f\in S$, we have  $\int_{\Omega}\varphi(x,f(x)) d\mu \leq 1.$
    Suppose that (\ref{condicion}) does not hold, so there exist $\varepsilon>0$, a scalar sequence $(\lambda_n)\searrow 0$ and a sequence $(f_n)$ in $S$ such that, for every $n\in\mathbb{N}$,
    \begin{equation}\label{cont}
    \int_{\Omega} \varphi(x, \lambda_n f_n(x))d\mu  > \lambda_n  \varepsilon
    \end{equation}
    and let us find a contradiction.

    Since $\varphi$ is an $N$-function, given $y\in\Omega$ and a natural $k$, for $\lambda$ small enough we have that $\frac{\varphi(y,k\lambda)}{k\lambda}\leq\frac{\varepsilon}{4k\cdot\mu(Z_k)}$. Hence $\mu(\{x\in Z_k:\frac{\varphi(x,k\lambda)}{k\lambda}>\frac{\varepsilon}{4k\cdot\mu(Z_k)}\})\xrightarrow{\lambda\rightarrow 0}0$. So, by taking a subsequence $(\lambda_{n_k})$ (and $(f_{n_k})$) denoted again $(\lambda_n)$ (and $(f_n)$), we can make the sets $A_n:=\{x\in Z_n: \frac{\varphi(x,n\lambda_n)}{\lambda_n}> \frac{\varepsilon}{4\mu(Z_n)}\}$ satisfy\footnote{of course, it can even be $A_n=\emptyset$ and $\mu(A_n)=0$.} $\mu(A_n)\leq \frac{\varepsilon}{4n}$.
    Thus (up to subsequences again in $(\lambda_n)$ and $(f_n)$) we can suppose that $(\lambda_n)$ verifies the properties (along with (\ref{cont})):
    \begin{equation}\label{coleccion}
    0\leq \lambda_n\leq \frac{1}{2n}, \quad
    \sum_{n=1}^{\infty} \lambda_n\leq 1, \quad
    \mu(A_n)\leq \frac{\varepsilon}{4n}, \quad
    \sup_{x\in Z_{n}\setminus A_n}\frac{\varphi(x,n \lambda_n)}{\lambda_n}\mu(Z_n)
    \leq\frac{\varepsilon}{4}.
    \end{equation}

    For the $N$-function $\varphi$, at each point $y\in\Omega$, we have its right-derivative $\varphi'(y,\cdot)$ satisfying $\varphi(y,t)=\int_0^t \varphi'(y,s) d\mu(s)$ (\cite{DHHR} page 53). For this right-derivative, Young's inequality (\ref{Young-ineq}) holds with equality (\cite{DHHR} Rmk 2.6.9), i.e. for every $y\in\Omega$, $t\geq 0$ and $u=\varphi'(y,t)$ we have $t\cdot u=\varphi(y,t)+\varphi^*(y,u)$. So, let us consider the functions $g_n(x):=\varphi'(x,\lambda_n f_n(x))$ which satisfy, for every $y\in \Omega$,
    \begin{equation}\label{Young-equality}
        \lambda_n f_n(y) g_n(y)= \varphi(y,\lambda_n f_n(y))+\varphi^*(y,g_n(y)).
    \end{equation}
    With this sequence define the function $g(t):=\sup_n \vert g_n(t)\vert$. We have that $g\in \lphiconj$. Indeed, $t\varphi'(x,t)\leq \varphi(x,2t)$ for every $x,t$ (\cite{DHHR} Lemma 2.6.6) applied to $\varphi^*(x,\varphi'(x,t))\leq t \varphi'(x,t)$ (\cite{DHHR} Eq. 2.6.15) gives
    $$\varphi^*(x,g_n(x))\leq \varphi(x,2\lambda_nf_n(x)).$$
    Then, using $\varphi(x,\lambda t)\leq \lambda \varphi(x,t)$ for $0<\lambda <1$ (\cite{DHHR} Eq. 2.3.8) and the above properties in (\ref{coleccion}), we have
    \begin{align*}
        \sup_{n\in\mathbb{N}}\int_{\Omega} \varphi^*(x, g_n(x)) d\mu
        {\leq}{}   &   \sum_{n=1}^\infty \int_{\Omega}  \varphi^*(x, g_n(x)) d\mu
        {\leq} \sum_{n=1}^\infty \int_{\Omega} \varphi(t,2\lambda_n f_n(x)) d\mu\\
        {\leq}{}   &    \sum_{n=1}^\infty 2\lambda_n\int_{\Omega} \varphi(x, f_n(x)) d\mu
        {\leq} \sum_{n=1}^\infty 2\lambda_n
        \leq  2.
    \end{align*}
    So the modular (and norm) of the sequence of functions $(\sup_{1\leq n\leq k} \vert g_n\vert)_k$ is bounded and, by Fatou property (\cite{DHHR} Thm 2.3.17 (d)), $g\in\lphiconj$ as $(\sup_{1\leq n\leq k} \vert g_n\vert)_k\nearrow g$ a.e.-$\mu$.

    Consider now the sets $B_n=\{x\in\Omega:\vert f_n(x)\vert>n\}$. 
    By Lemma \ref{mallema1} we have $\mu(B_n)\xrightarrow{n\rightarrow\infty}0$.
    Then, by (\ref{basico}) in Theorem \ref{clave}, there exists a natural $n_1$ such that, for every $n\geq n_1$,
    $$
    \int _{B_n}\vert f_n g\vert d\mu <\frac{\varepsilon}{4}.
    $$
    Also, by (\ref{fuerafinito}) in the same Theorem \ref{clave}, there exists a natural $n_2$ such that, for every $n\geq n_2$,
    $$
    \int _{\Omega\setminus Z_n}\vert f_n g\vert d\mu <\frac{\varepsilon}{4}.
    $$
    Indeed, there exists $A$ with $\mu(A)<\infty$ such that, for all $n\in\mathbb{N}$,
    $$
    \int_{\Omega\setminus A} \vert f_n g \vert d\mu<\frac{\varepsilon}{8}.
    $$
    Now, consider $(A\cap Z_n)\nearrow A$, so $\mu(A\setminus(A\cap Z_n))\rightarrow 0$ as $n\rightarrow\infty$. Hence,
    $$
    \int_{\Omega\setminus(A\cap Z_n)} \vert f_n g\vert d\mu
    =\int_{\Omega\setminus A} \vert f_n g\vert d \mu + \int_{A\setminus(A\cap Z_n)} \vert f_n g \vert d\mu.
    $$
    As $\mu(A\setminus(A\cap Z_n))\rightarrow 0$, by (\ref{basico}) in Theorem \ref{clave}, there exists this natural $n_2$ such that, for $n\geq n_2$,
    $$
    \int_{A\setminus A\cap Z_n} \vert f_n g\vert d\mu <\frac{\varepsilon}{8}.
    $$
    Thus, for $n\geq n_2$,
    $$
    \int_{\Omega\setminus Z_n} \vert f_n g\vert d\mu
    \leq \int_{\Omega\setminus (A\cap Z_n)} \vert f_n g\vert d\mu
    =\int_{\Omega\setminus A} \vert f_n g\vert d \mu + \int_{A\setminus(A\cap Z_n)} \vert f_n g \vert d\mu
    =\frac{\varepsilon}{8}+\frac{\varepsilon}{8}=\frac{\varepsilon}{4}.
    $$

    Therefore, using all of the above including (\ref{Young-equality}) and (\ref{coleccion}) we conclude that, for $n\geq \max\{n_1,n_2\}$,
    \begin{align*}
        \int_{\Omega} \varphi(x,\lambda_n f_n(x)) d\mu {}    &   \leq\int_{B_n} \varphi(x,\lambda_n f_n(x)) d\mu + \int_{\Omega\setminus Z_n} \varphi(x,\lambda_n f_n(x)) d\mu + \int_{Z_n\setminus B_n} \varphi(x,\lambda_n f_n(x)) d\mu\\
        {}  &   \leq \int_{B_n} \vert \lambda_n f_n(x) g_n(x)\vert d\mu + \int_{\Omega\setminus Z_n} \vert \lambda_n f_n(x) g_n(x)\vert d\mu + \int_{Z_n\setminus B_n} \varphi(x,\lambda_n n) d\mu\\
        {}  &  \leq \lambda_n \frac{\varepsilon}{4} +\lambda_n \frac{\varepsilon}{4} + \int_{(Z_n\setminus B_n)\cap A_n} \varphi(x,\lambda_n n) d\mu + \int_{(Z_n\setminus B_n)\setminus A_n} \varphi(x,\lambda_n n) d\mu\\
        {}  &   \leq \lambda_n \frac{\varepsilon}{2} + \sup_{x\in A_n} \varphi(x,\lambda_n n) \mu(A_n) + \sup_{x\in Z_n\setminus A_n} \varphi(x,\lambda_n n) \mu(Z_n)\\
        {}  &   \leq \lambda_n \frac{\varepsilon}{2} + \lambda_n n \frac{\varepsilon}{4n} + \lambda_n\frac{\varepsilon}{4}\\
        {}  &   =\lambda_n  \varepsilon,
    \end{align*}
    which is a contradiction with (\ref{cont}).

$(\Leftarrow):$
    Let us first show that $S$ is bounded. If not, given any $0<\lambda\leq 1$, there would exist $f_\lambda\in S$ such that $\lVert f_\lambda\rVert_{\varphi} > \frac{1}{\lambda}$ or, equivalently, $\rho_{\varphi}(\frac{f_\lambda}{\frac{1}{\lambda}})=\rho_{\varphi}(\lambda f_\lambda) >1$. Thus, for every $0<\lambda<1$,
    $$
    \sup_{f\in S}\frac{1}{\lambda}\int_{\Omega} \varphi(x,\lambda f(x)) d\mu
    =\sup_{f\in S}\frac{1}{\lambda}\rho_{\varphi}(\lambda f)
    \geq\sup_{f\in S}\frac{1}{\lambda}\rho_{\varphi}(\lambda f_\lambda)
    > \frac{1}{\lambda}>1,
    $$
    which contradicts (\ref{condicion}).

    Now, let $g\in \lphiconj$ and  $r>0$  such that $\int_{\Omega} \varphi^*(x, rg(x))d\mu<\infty$. By hypothesis, given $\varepsilon>0$, there exists $\lambda_0>0$ such that
    $$
    \sup_{f\in S}\frac{1}{\lambda_0}\int_{\Omega} \varphi(x, \lambda_0 f(x)) d\mu < \frac{\varepsilon r}{2}.
    $$
    Take $\delta>0$ such that, for every measurable set $E$ with $\mu(E)<\delta$,
    $$\int_E \varphi^*(x, rg(x)) d\mu <\frac{\lambda_0 r \varepsilon}{2},$$
    and take $A\subset \Omega$ with $\mu(A)<\infty$ big enough such that
    $$\int_{\Omega\setminus A} \varphi^*(x, rg(x))d\mu <\frac{\lambda_0 r \varepsilon}{2}.$$
    Then, using  Young's inequality (\ref{Young-ineq}),  we have
    $$
    \sup_{f\in S}\int_E\vert f(x)g(x)\vert d\mu
    \leq \frac{1}{\lambda_0 r}\left[\sup_{f\in S}\int_E \varphi(x, \lambda_0 f(x)) d\mu + \int_E \varphi^*(x, rg(x)) d\mu \right]
    < \frac{1}{r}(\frac{\varepsilon r}{2})+\frac{1}{\lambda_0 r}(\frac{\lambda_0 r \varepsilon}{2})
    =\varepsilon,
    $$
    as well as
    $$
    \sup_{f\in S}\int_{\Omega\setminus A}\vert f(x)g(x)\vert d\mu
    \leq \frac{1}{\lambda_0 r}\left[\sup_{f\in S}\int_{\Omega\setminus A} \varphi(x, \lambda_0 f(x)) d\mu + \int_{\Omega\setminus A} \varphi^*(x, rg(x)) d\mu \right]
    < \frac{1}{r}(\frac{\varepsilon r}{2})+\frac{1}{\lambda_0 r}(\frac{\lambda_0 r \varepsilon}{2})
    =\varepsilon.
    $$
    Thus, applying Theorem \ref{clave}, we conclude that $S$ is relatively weakly compact in $\lphi$.
\end{proof}


In particular, for $(\Omega,\mu)$ be $\mathbb{N}$ with the counting measure, we have:

\begin{corolario}\label{corolario-Nakano1}
Let $\ell_{(\varphi_n)}$ be a Musielak-Orlicz sequence space for a proper constrained generalized $N$-function $(\varphi_n)$ satisfying the $\Delta_2$-condition. A subset $S\subset \ell_{(\varphi_n)}$ is relatively weakly compact if and only if
$$
\lim_{\lambda\rightarrow 0} \sup_{x=(x_n)\in S} \frac{1}{\lambda}\rho_{(\varphi_n)}(\lambda x)
=\lim_{\lambda\rightarrow 0} \sup_{x=(x_n)\in S} \frac{1}{\lambda}\sum_{n=1}^{\infty} \varphi_n(\lambda x_n)=0.
$$
\end{corolario}

A consequence of Theorem \ref{crit2} is the weakly Banach-Saks property under the extra hypothesis that $\lphi$ satisfies the subsequence splitting property. Recall that a Banach space $X$ is said to be weakly Banach-Saks if every weakly convergent sequence $f_n\xrightarrow{w} f$ has a Cèsaro convergent subsequence $(f_{n_k})$, i.e.
$$
\left\lVert\frac{f_{n_1}+...+f_{n_k}}{k}-f\right\rVert_X\xrightarrow{k\rightarrow \infty}0.
$$

\begin{teorema}\label{wBS2}
Let $\lphi$ be a Musielak-Orlicz space for a $\sigma$-finite measure space $(\Omega,\mu)$ and a proper constrained generalized $N$-function $\varphi$ satisfying the $\Delta_2$-condition. If $\lphi$ has the subsequence splitting property, then it is weakly Banach-Saks. 
\end{teorema}

\begin{proof}
Due to Corollary 3.4 in \cite{FR}, it is enough to prove the weak Banach-Saks property for disjoint sequences. So, assume that $(f_n)$ is a pairwise disjoint weakly convergent (to $f=0$) sequence in $\lphi$. We need to prove that $(f_{n_{k}})$ is Ces\`aro convergent (to $0$) for some subsequence $(f_{n_{k}})$. Since $\varphi$ satisfies the $\Delta_2$-condition, it is enough to prove the modular convergence. As $(f_n)$ is a weakly convergent sequence, it is a relatively weakly compact set, so, by Theorem \ref{crit2}, we have
$$
\lim_{\lambda\rightarrow 0} \sup_{n} \frac{\rho_{\varphi}(\lambda f_n)}{\lambda}=0.
$$
Hence, as $\rho_{\varphi}(g+h)=\rho_{\varphi}(g)+\rho_\varphi(h)$ for disjoint functions $g$ and $h$, we get
\begin{align*}
    \lim_{n\rightarrow\infty}\rho_{\varphi}\left(\frac{f_1+...+f_n}{n}\right)
    ={} &   \lim_{n\rightarrow\infty}\sum_{i=1}^n \rho_{\varphi}\left(\frac{f_i}{n}\right)
    \leq\lim_{n\rightarrow\infty}\sum_{i=1}^n \sup_{1\leq k\leq n}\rho_{\varphi}\left(\frac{f_k}{n}\right)\\
    ={} &   \lim_{n\rightarrow\infty} \sup_{1\leq k\leq n}\left(n\cdot\rho_{\varphi}\left(\frac{f_k}{n}\right)\right) = 0.
\end{align*}
So, the same sequence $(f_n)$ is also Cèsaro convergent to $0$.
\end{proof}

\begin{nota}\label{Nakano-wBS}
    For Musielak-Orlicz sequence spaces it was already known that the weak Banach-Saks property is equivalent to the separability of $\ell_{(\varphi_n)}$ (\cite{Kaminska-Lee} Cor 2.2).
\end{nota}

The second And\^o's type characterization of relatively weakly compact sets in $\lphi$ is given in terms of being bounded subsets in a suitable Musielak-Orlicz spaces $L^{\psi}(\Omega)$.
The following definition was presented in \cite{PCM3}, generalizes the one given by And\^{o} (\cite{Ando}) for Young functions to the class of generalized $\Phi$-functions and is more restrictive than the one introduced in \cite{PCM}.

\begin{definicion}
A generalized $\Phi$-function $\psi(x,t)$ increases uniformly more rapidly than another generalized $\Phi$-function $\varphi(x,t)$ for all\footnote{We use the wording ``for all $t$'' in contrast with ``for large $t$'', which is the silent definition used in \cite{Ando}. Nowak already distinguished between them in \cite{Nowak}, and the distinction is relevant for spaces $\Omega$ of infinite measure.} $t$ if for each $\varepsilon>0$ there exists some $\delta>0$ such that, for all $t\geq 0$ and a.e. $x\in \Omega$,
$$
\varepsilon\psi(x,t) \geq  \frac{1}{\delta}\varphi(x,\delta t).
$$
Or, equivalenty, that for every $d>0$ there exists $s>0$ such that, for all $\overline{t}\geq 0$ and a.e. $x\in\Omega$,
$$
\frac{\psi(x,s \overline{t})}{s}\geq d\cdot\varphi(x,\overline{t}).
$$
\end{definicion}


\begin{teorema}\label{Th.M-O}
Let $\lphi$ be a Musielak-Orlicz space for a $\sigma$-finite measure space $(\Omega,\mu)$ and a proper constrained generalized $N$-function $\varphi$ satisfying the $\Delta_2$-condition. A subset $S\subset\lphi$ is relatively weakly compact if and only if there exists a generalized $\Phi$-function $\psi(x,t)$ increasing uniformly more rapidly than $\varphi(x,t)$ for all $t$ such that $S$ is norm bounded in $L^{\psi}(\Omega)$ (even $S\subset B_{L^{\psi}(\Omega)}$).
\end{teorema}
\begin{proof}
$(\Leftarrow)$: Suppose w.l.o.g. that $S\subset B_{L^{\psi}(\Omega)}$, hence $\rho_{\psi}(f)\leq 1$ for every $f\in S$. 
Given $\varepsilon>0$, there exists $\delta>0$ such that, for all $t\geq 0$ and a.e.-$x\in\Omega$,
$$
\varepsilon\psi(x,t) \geq  \frac{1}{\delta}\varphi(x,\delta t).
$$
Then, for each $0<\lambda\leq \delta$,
\begin{align*}
\sup_{f\in S} \frac{1}{\lambda}\int_{\Omega} \varphi(x,\lambda f(x))d\mu
\leq \sup_{f\in S} \frac{1}{\delta}\int_{\Omega} \varphi(x,\delta f(x)) d\mu
\leq \sup_{f\in S}\int_{\Omega} \varepsilon\psi(x,f(x))d\mu
\leq{} \varepsilon.
\end{align*}
So, using Theorem \ref{crit2}, we deduce that $S$ is relatively weakly compact in $\lphi$.

$(\Rightarrow)$:
Assume that $S$ is relatively weakly compact. By Theorem \ref{crit2}, there exists a sequence $(\lambda_n)\searrow 0$ with
\begin{equation}\label{}
\sup_{f\in S}\frac{1}{\lambda_n}\int_{\Omega} \varphi(x, \lambda_n f(x)) d\mu \leq  \frac{1}{2^{2n}}
\end{equation}
for every $n\in\mathbb{N}$. Let us define the function
$$
\psi(x,t):=\sum_{n=1}^{\infty} \frac{2^n}{\lambda_n}\varphi(x, \lambda_n t).
$$

$\psi(x,t)$ is a Musielak-Orlicz function increasing uniformly more rapidly than $\varphi(x,t)$ for all $t$. Indeed, since $\psi(x,t)\geq \frac{2^i}{\lambda_i}\ \varphi(x,\lambda_i t)$ for each natural $i$, given $\varepsilon>0$, we take $n$ so that  $2^n\geq \frac{1}{\varepsilon}$ and $\delta=\lambda_n$, getting
$$\varepsilon\psi(x,t)\geq \frac{1}{\delta} \varphi(x,\delta t).$$
Furthermore, for each $f\in S$, by Beppo-Levi Theorem we get
$$
\int_{\Omega} \psi(x,f(x))d\mu
=\;  \sum_{n=1}^{\infty}\frac{2^n}{\lambda_n}\int_{\Omega} \varphi(x, \lambda_n f(x)) d\mu \;
\leq\;  \sum_{n=1}^{\infty} \frac{2^n}{2^{2n}}
= 1,
$$
so $S\subset B_{L^{\psi}(\Omega)}$, completing the proof.

\end{proof}

\begin{nota}
    The function $\psi$ constructed above can be shown or modified to be finite a.e. $x\in\Omega$ and for all $t\geq 0$. The idea is that, for $y\in\Omega$ with $\psi(y,f(y))<\infty$ for some $f\in S$, $\psi(y,t)<\infty$ for $0\leq t\leq f(y)$ and also $\psi(y,t)<\infty$ for $t> f(y)$ becuase of the $\Delta_2$-condition of $\varphi$; and for every other $y\in\Omega$ there is no interaction with $S$ so we can redefine $\psi$ with any values that satisfy the hypothesis (see \cite{PCM3}, \cite{Ando}, \cite{Nowak}).
\end{nota}


\begin{corolario}
Let $\ell_{\varphi_n}$ be a Musielak-Orlicz sequence space for a proper constrained generalized $N$-function $(\varphi_n)$ satisfying the $\Delta_2$-condition. A subset $S\subset \ell_{\varphi_n}$ is relatively weakly compact if and only if there exists a sequence of Young functions $(\psi_n)$ increasing uniformly more rapidly than $(\varphi_n)$ such that $S$ is bounded in $\ell_{\psi_n}$.
\end{corolario}

\section{Weakly convergent sequences}\label{sec.wconvergencia}

We now give three criteria for the weak convergence of sequences in $\lphi$ inherited from the previous Theorems \ref{clave}, \ref{crit2} and \ref{Th.M-O}.

\begin{proposicion}\label{sucesiones1}
Let $\lphi$ be a Musielak-Orlicz space for a $\sigma$-finite measure space $(\Omega,\mu)$ and a proper constrained generalized $N$-function $\varphi$ satisfying the $\Delta_2$-condition and let the sequence $(f_{n})$ and $f$ in $\lphi$. Then, $f_{n}\rightarrow f$ weakly if and only if all the three
\begin{itemize}
\item [($i$)] $\lim_{n} \int_{A} (f_{n}-f) d\mu = 0$ for each $A\in \Sigma$ with $\mu(A)<\infty$,
    \item [($ii$)] $\lim_{\mu(A)\rightarrow 0} \sup_{n} \int_A \vert (f_{n}-f) g\vert d\mu = 0$ for each function $g\in L^{\varphi^{*}(\cdot)}(\Omega)$, and
    \item [($iii$)] For each function $g\in L^{\varphi^{*}(\cdot)}(\Omega)$ and for every $\varepsilon>0$, there exists $A\in \Sigma$ with $\mu(A)<\infty$ such that $\, \sup_{n} \int_{\Omega\setminus A} \vert (f_{n}-f) g\vert d\mu \leq\varepsilon$.
\end{itemize}
\end{proposicion}

\begin{proof}  $(\Rightarrow):$ Clearly, $(i)$ holds since $\chi_{A} \in L^{\varphi^{*}(\cdot)}$ as $\varphi$ is proper, and conditions $(ii)$ and $(iii)$ follow from above Theorem \ref{clave}, as the sequence $(f_n-f)$ is a weakly compact set.

$(\Leftarrow):$
    We can assume w.l.o.g.  $f=0$. We must show that, for every $g\in (L^{\varphi}(\Omega))^*= L^{\varphi^*}(\Omega)$, $\lim_n \int f_ng d\mu=0$. If  $g\in L^{\varphi^{*}(\cdot)}(\Omega)$ is a simple function with finite measure support, then it follows directly from $(i)$ that $\lim_{n} \int_{\Omega} f_{n}g  d\mu = 0$.
    
    Assume now that $g\in L^{\varphi^*(\cdot)}(\Omega)$ is a bounded function.
    Given $\varepsilon > 0$, by condition $(iii)$ there exists $A\subset \Omega$ with $\mu(A)<\infty$ such that
    $$
    \sup_{n} \int_{\Omega\setminus A} \vert f_{n} g\vert d\mu \leq\frac{\varepsilon}{3}
    $$
    and we can take a simple function $g_{s}$ such that $\lVert(g-g_{s})\chi_{A}\rVert_{\infty} < \frac{\varepsilon}{3}$. Thus,
    \begin{align*}
    \int_{\Omega} |f_{n} g| d\mu
    \leq{}  &   \int_{A} |f_{n} (g-g_{s})| d\mu + \int_{A} |f_{n}g_{s}| d\mu + \int_{\Omega\setminus A} \vert f_n g\vert d\mu \\
    \leq{}  &   \frac{\varepsilon}{3} \int_{A} |f_{n}| d\mu + \int_{A} |f_{n}g_{s}| d\mu + \frac{\varepsilon}{3}
    \end{align*}
    and hence $\int_{\Omega} |f_{n} g| d\mu  \leq \varepsilon$ for $n$ large enough.

    Now, for an arbitrary $g\in L^{\varphi^{*}(\cdot)}(\Omega)$, by condition $(ii)$, there exists $\delta > 0$ such that $\int_{A} |f_{n} g| d\mu < \frac{\varepsilon}{2}$ if $\mu(A) < \delta$ for every natural $n$. Consider $G_{m}= \{t \in \Omega: |g(t)|\leq m\}$ with $m$ large enough so that $\mu(\Omega\setminus G_{m}) < \delta $. Then,
    $$
    \int_{\Omega} |f_{n} g| d\mu
    =\int_{G_{m}} |f_{n} g| d\mu + \int_{\Omega\setminus G_{m}} |f_{n} g | d\mu
    \leq  \int_{\Omega} |f_{n} g \chi_{G_{m}}| d\mu + \frac{\varepsilon}{2}.
    $$
    Now, since $g\chi_{G_m}$ is bounded, we deduce $\int_{\Omega} |f_{n} g| d\mu \leq \varepsilon$ for $n$ large enough.
\end{proof}

\begin{proposicion}\label{sucesiones2}
Let $\lphi$ be a Musielak-Orlicz space for a $\sigma$-finite measure space $(\Omega,\mu)$ and a proper constrained generalized $N$-function $\varphi$ satisfying the $\Delta_2$-condition and let a sequence $(f_{n})$ and $f$ in $\lphi$. Then, $f_n\rightarrow f$ weakly if and only if both
\begin{itemize}
    \item [$(i)$] $\lim_{n} \int_{A} f_{n} d\mu  = \int_{A} f  d\mu$ for each $A\in \Sigma$ with $\mu(A)<\infty$, and
    \item [$(ii)$] $\lim_{\lambda\rightarrow 0}\sup_{n} \frac{1}{\lambda}\int_{\Omega} \varphi(x, \lambda ( f_{n}-f) ) d\mu = 0$.
\end{itemize}
\end{proposicion}

\begin{proof}
Clearly, $(i)$ is equivalent to condition $(i)$ in Proposition \ref{sucesiones1} and, applying Therorem \ref{crit2} to the weakly convergent sequence $(f_n-f)$, we get condition $(ii)$.

Conversely, let $g\in L^{\varphi^*(\cdot)}(\Omega)$ and  $r>0$  such that $\int_{\Omega} \varphi^*(x, rg(x))d\mu<\infty$. By hypothesis $(ii)$, given $\varepsilon>0$, there exists $\lambda_0>0$ such that
$$
\sup_{n}\frac{1}{\lambda_0}\int_{\Omega} \varphi(x, \lambda_0 (f_n(x)-f(x))) d\mu < \frac{\varepsilon r}{2}.
$$
Take $A\subset \Omega$ with $\mu(A)<\infty$ such that
$$
\int_{\Omega\setminus A} \varphi^*(x, rg(x))d\mu <\frac{\varepsilon}{2}
$$
and let $\delta>0$  such that, for every measurable set $E$ with $\mu(E)<\delta$,
$$\int_E \varphi^*(x, rg(x))d\mu <\frac{\varepsilon \lambda_0 r}{2}.$$
Thus, using  Young's inequality (\ref{Young-ineq}),  we have:
\begin{align*}
\sup_{n}\int_{\Omega\setminus A} {}  &   \vert (f_n(x)-f(x))g(x)\vert d\mu \\
 \leq{}  &   \frac{1}{\lambda_0 r}\left[\sup_{n}\int_{\Omega\setminus A} \varphi(x, \lambda_0 (f_n(x)-f(x))) d\mu + \int_{\Omega\setminus A} \varphi^*(x, rg(x)) d\mu \right]\\
<{}  &   \frac{1}{r}(\frac{\varepsilon r}{2})+\frac{\varepsilon}{2}
=\varepsilon,
\end{align*}
as well as
\begin{align*}
\sup_{n}\int_E {}  &   \vert (f_n(t)-f(t))g(t)\vert d\mu \\
\leq{}  &   \frac{1}{\lambda_0 r}\left[\sup_{n}\int_E \vert \lambda_0 (f_n(t)-f(t))\vert^{p(t)} d\mu + \int_E \vert rg(t)\vert^{p^{*}(t)} d\mu \right] \\
<{}  &   \frac{1}{r}(\frac{\varepsilon r}{2})+\frac{1}{\lambda_0 r} (\frac{\varepsilon \lambda_0 r}{2})
=\varepsilon.
\end{align*}
Thus, conditions $(iii)$ and $(ii)$ of the above Proposition \ref{sucesiones1} are satisfied and we conclude that $(f_n)$ is weakly convergent to $f$.
\end{proof}

\begin{proposicion}\label{sucesiones3}
Let $\lphi$ be a Musielak-Orlicz space for a $\sigma$-finite measure space $(\Omega,\mu)$ and a proper constrained generalized $N$-function $\varphi$ satisfying the $\Delta_2$-condition and let a sequence $(f_{n})$ and $f$ in $\lphi$. Then, $f_n\rightarrow f$ weakly if and only if both
\begin{itemize}
    \item [$(i)$] $\lim_{n} \int_{A} f_{n} d\mu  = \int_{A} f  d\mu$ for each $A\in \Sigma$ with $\mu(A)<\infty$, and
    \item [$(ii)$] $S=\{f_n:n\in\mathbb{N}\}\cup \{f\}$ is bounded in $\lpsi$ for some generalized $\Phi$-function $\psi(x,t)$ increasing uniformly more rapidly than $\varphi(x,t)$ for all $t$.
\end{itemize}
\end{proposicion}

\begin{proof}
    Again, $(i)$ is the same as in Propositions \ref{sucesiones1} and \ref{sucesiones2}, and applying Theorem \ref{Th.M-O} to the weakly convergent amplied sequence $(f_n)$ with the addition of $f_0=f$ we get $(ii)$. Conversely, if $S$ is norm bounded in $\lpsi$, so is the sequence $(f_n-f)$. Hence, by Theorem \ref{Th.M-O} $(f_n-f)$ is a weakly compact subset and by Theorem \ref{clave} we have $\lim_{\lambda\rightarrow 0}\sup_{n} \frac{1}{\lambda}\int_{\Omega} \varphi(x, \lambda ( f_{n}-f) ) d\mu = 0$, i.e. $(ii)$ in Theorem \ref{sucesiones2}. concluding that $f_n\rightarrow f$ weakly.
\end{proof}

In particular, it follows that in reflexive $\lphi$ spaces (with the necessary conditions on $\varphi$), a sequence $(f_{n})$ is weakly convergent to  $f\in \lphi $ if and only if $(f_{n})$ is norm bounded and $\int_{A} f_{n} d\mu \rightarrow \int_{A} f d\mu$, for every measurable $A \subset \Omega$ with $\mu(A)<\infty$.

Let us end with the enunciations of the corresponding results in Musielak-Orlicz sequence spaces.

\begin{corolario}\label{corolario-Nakano2}
Let $\ell_{\varphi_n}$ be a Musielak-Orlicz sequence space for a generalized $N$-function $(\varphi_n)$ with the $\Delta_2$ property. A sequence $(x^k)_k$  in $\ell_{\varphi_n}$ converges weakly to $y=(y_n)\in \ell_{\varphi_n}$ if and only if both
\begin{itemize}
    \item [$(i)$] For each coordinate $n$, $\lim_{k\rightarrow \infty} x_n^k=y_n$, and
    \item [$(ii)$] any one of the following the conditions is satisfied, either   
    \begin{itemize}
        \item [$(a)$] For every $(y_n)\in\ell_{\varphi_n^*}$ and $\varepsilon>0$, there exists a large enough natural $n_0$ such that $\sup_{k\in\mathbb{N}}\sum_{n=n_0}^{\infty} \vert(x^k_n-x)\cdot y_n\vert\leq \varepsilon$,
        \item [$(b)$] $\lim_{\lambda\rightarrow 0}\sup_{k\in\nat} \frac{1}{\lambda}\sum_{n=1}^{\infty} \varphi_n (\lambda (x_n^k-y_n))=0$, or
        \item [$(c)$] $(x^k)$ is norm bounded in $\ell_{\psi_n}$ for a sequence of Orlicz functions $(\psi_n)$ increasing uniformly more rapidly than $(\varphi_n)$ for all $t$.
    \end{itemize}
\end{itemize}
\end{corolario}

\section{Final comments}

Although the text presented here is complete, it will not be submitted to a journal in the very short term, since we think that the article could be improved in significant ways. We want to bypass Lemma \ref{mallema1} and be able to prove Theorem \ref{crit2} without it, allowing us to remove the hypothesis of $\varphi$ being constrained along the article. 

Also, it is probably already known whenever $\lphi$ has the subsequence splitting property (may be the $\Delta_2$-condition is enough), so the weakly Banach-Saks property would be proved with more clear hypothesis.

Finally, in \cite{Ando} and \cite{Nowak}, it is said that a Young function $\varphi$ increases uniformly more rapidly than itself if and only if it satisfies the $\nabla_2$-condition ($\Delta_2$ for $\varphi^*$). For generalized $\Phi$-functions it should hold the same, as $\Delta_2$ plus $\nabla_2$ implies that $\lphi$ is reflexive and hence weak compactness is equivalent to boundedness. But we have not been able to prove it yet (neither for Young functions, as And\^o and Nowak do not prove and we fear we are missing an easy argument), and neither have found clear references (or proved) stating that $\nabla_2$ (considered as $\varphi(x,2t)\geq K\varphi(x,t)$) is equivalent to $\varphi^*$ having the $\Delta_2$-condition.

\end{document}